\documentclass[12pt]{amsart}
\usepackage{amsmath}
\usepackage{amsfonts}
\usepackage{amssymb}
\usepackage{mathrsfs}
\usepackage{appendix}
\usepackage{amsthm}
\usepackage{geometry}
\usepackage{txfonts}
\usepackage{graphicx}

\usepackage[alphabetic]{amsrefs}
\usepackage{hyperref}
\usepackage{color}
\usepackage{verbatim}

\allowdisplaybreaks
\theoremstyle{definition}

 \numberwithin{equation}{section}

\newcommand{\dv}{\mathrm{div}\,}

\newtheorem{Theorem}{Theorem}[section]
\newtheorem{Lemma}{Lemma}[section]
\newtheorem{Proposition}{Proposition}[section]
\newtheorem{Def}{Definition}[section]
\newtheorem{Remark}{Remark}[section]
\newtheorem{Corollary}{Corollary}[section]
\newtheorem{Conv}{Convention}[section]
\newtheorem{Example}{Example}[section]

\newcommand{\R}{\mathbb{R}}
\newcommand{\C}{\mathcal{C}}
\newcommand{\D}{\mathbb{D}}

\newcommand{\U}{\mathscr{U}}

\newcommand{\PP}{\mathcal{P}}
\newcommand{\BB}{\mathcal{B}}
\newcommand{\CC}{\mathcal{C}}

\newcommand{\HH}{\mathscr{H}}

\newcommand{\LL}{\mathcal{L}}

\newcommand{\s}{\mathcal{S}}

\newcommand{\pv}{\text{P.V.}}
\newcommand{\lap}{\Delta}

\newcommand{\rcapw}{{\rm Cap}_{p,|t|^\theta}}
\newcommand{\rcapb}{{\rm Cap}_{B^s_p}}
\newcommand{\rcap}{{\rm Cap}}

\newcommand{\diam}[1]{\textrm{diam}({#1})}

\newcommand{\LHMD}{\rm LHMD}

\newcommand{\LHMDs}{{\rm LHMD}^s}
\newcommand{\GHMDs}{{\rm GHMD}^s}

\newcommand{\mnote}[1]{ \marginpar{\color{red}\tiny\em #1}}

\begin{document}
\title[H\"older extension for fractional Laplacian]
 {H\"older extension for fractional Laplacian}

\author{Feng LI}

\address{Department of Mathematics, Uppsala University, Sweden}
\address{Department of Mathematics, University of Cincinnati, USA}
\email{lifenggoo@gmail.com}

\begin{abstract}
In this note, we characterize the sharp boundary condition such that the fractional harmonic extensions with H\"older regularity up to the boundary is globally H\"older continuous. The proofs are based on estimates of fractional harmonic measure decay and uniform fractional fatness of the complement of the domain.
\end{abstract}


\maketitle

\tableofcontents

\section{Motivation and purpose}

Let $s\in(0,1)$, $N\geq 2$, and consider the fractional elliptic equation
\begin{equation}\label{equ:main}
    \left\{
    \begin{split}
        &(-\lap)^su=0\quad in\ \Omega\\
        &u=g\quad in\ \R^N\setminus\Omega,
    \end{split}
    \right.
\end{equation}
in a bounded domain $\Omega\in\R^N$, where 
$$
(-\lap)^su(x)=c_{N,s}\pv\int_{\R^N}\frac{u(x)-u(y)}{|x-y|^{N+2s}}\,dy,
$$
where $c_{N,s}$ is a normalization constant.

The solution to the Dirichlet boundary value problem in a bounded domain corresponds with the minimizer of the variational integral
$$
J(u)=\frac{1}{2}\int_{\R^N}\int_{\R^N}\frac{\left|v(x)-v(y)\right|^2}{|x-y|^{N+2s}}\,dxdy
$$
among all functions $v-g\in W^{s,2}_0(\Omega)$, which is defined as a completion of $C^{\infty}_0(\Omega)$ under the norm $W^{s,2}(\R^N)$. 
As a non-local operator, the Dirichlet boundary data for \ref{equ:main} are prescribed on the complement $\R^N\setminus\Omega$, rather than just on the boundary $\partial\Omega$.

Given a function $g$ on $\R^N\setminus\Omega$, we denote by $P_\Omega g$ the Perron Dirichlet solution of $g$ over $\Omega$. 
A boundary point $\xi\in\partial\Omega$ is called {\it regular} if $\lim\limits_{\Omega\ni x\rightarrow\xi}P_\Omega(x)=g(\xi)$ if for every function $g\in C(\R^N)\cap L^{1}_{2s}(\R^N)$, in which the space $L^{p-1}_{sp}(\R^N)$ is the Tail space in nonlocal problems (see e.g. \cite{KKP}). 
Thus if $\Omega$ is regular, then $P_\Omega$ maps $C(\R^N\setminus\Omega)\cap L^{1}_{2s}(\R^N)$ to $\HH^{2s}(\Omega)\cap C(\overline{\Omega})$, where $\HH^{2s}(\Omega)$ is the family of fractional harmonic functions on $\Omega$ with respect to operator $\left(-\Delta\right)^s$. 
From the point of view of the Wiener criterion, the regularity of a boundary point can be characterized using capacity. For the Wiener criterion in a non-local setting, one can refer to very recent results in \cite{Bjorn, KLL}.

Then it's natural to consider whether the better regularity of a boundary function $g$ triggers the better regularity of $P_\Omega g$ and vice versa. If the answer is positive, what would the regularity of boundary look like as the continuity criterion in the Wiener sense? 

In quasilinear elliptic equations, Maz`ya in \cite{Mazya2} and Gariepy $\&$ Ziemer in \cite{GZ} proved that the uniform capacity density condition on regular points would yield the H\"older regularity up to these regular points. For the weighted measure with respect to $p$-harmonic functions one can refer to Theorem 6.44 in \cite{HKM}, in which only the sufficiency of uniformly $p$-fat was considered. While in \cite{LZLH} the authors gave a sufficient condition for boundary H\"older regularity via Lebesgue measure density for elliptic equations and fractional Laplacians. For the almost optimal boundary regularity of fractional operators on smooth enough $\Omega$, we refer the reader to \cite{RS, RS2, IMS}. In our note, we give the critical regularity condition of $\Omega$ to guarantee the H\"older extension on $\Omega$ with H\"older regular boundary of fractional harmonic functions via the equivalence between fractional harmonic measure decay and uniform fractional capacity density (or capacity fatness) of the regular boundary points.

In fact, for classical harmonic functions, Aikawa proved some equivalent conditions for the H\"older extension on regular Euclidean domains in \cite{Aikawa}, and then in \cite{AS} Aikawa and Shanmugalingam extended this property to $p$-harmonic functions  in metric spaces. In this note, we extend the results to the fractional case. More
precisely we study the H\"older extension property in the context of fractional Laplacians with H\"older boundary data $g$. Let $0<\beta\leq\alpha\leq\sigma_0<1$. 
Consider the family $\Lambda_\alpha(\Omega)$ of all bounded $\alpha$-H\"older continuous functions $h$ on $\Omega$ with norm
$$
\|h\|_{\Lambda_\alpha(\Omega)}:=\sup\limits_{x\in\Omega}|h(x)|+\mathop{\sup\limits_{x,y\in\Omega}}\limits_{x\neq y}\frac{|h(x)-h(y)|}{|x-y|^\alpha}<\infty.
$$
We are concerned about the finiteness of the operator norm:
\begin{equation}\label{op:fracmap}    \|P_\Omega\|_{\alpha\rightarrow\beta}:=\mathop{\sup_{h\in\Lambda_\alpha(\R^N\setminus\Omega)}}\limits_{\|h\|_{\Lambda_\alpha(\R^N\setminus\Omega)}\neq 0}\frac{\|P_\Omega h\|_{\Lambda_\beta(\Omega)}}{\|h\|_{\Lambda_{\alpha}(\R^N\setminus\Omega)}}.
\end{equation}

\quad

\section{(Fractional) Perrons solutions and (fractional) harmonic measures}
In this note, we utilize Perron solutions and Caffarelli-Silvestre extension to define fractional harmonic measure. For recent results on Perron solutions for fractional $p$-Laplacians one can refer to \cite{KKP, LL}. However it's more direct to define fractional harmonic measure by Poisson kernel (see e.g. \cite{Bogdan, Wu, CaSi}), we don't use this approach in this note.

\subsection{Weighted Sobolev spaces, Perron solutions, and harmonic measure}
In this section, we mainly discuss the Perron solution associated with degenerate weighted elliptic equations, also the weighted $p$-harmonic measure, derived from Perron solutions (refer to \cite{AS, HKM}). As we know, when $p\neq 2$, the harmonic measure is already not a measure due to the nonlinearity. Although we only use the case of $p=2$, the properties on general $p$ have their own interests.

Let $E\subset\R^{N+1}$ be open. Let Radon measure $\mu$ be $p$-admissible (see \cite{HKM}, section 1.1), associated with a locally integrable, nonnegative function $w$ in $\R^{N+1}$, define by
\begin{equation}\label{def-mu}
    \mu(E)=\int_E w(x)\,dx.
\end{equation}
Thus $d\mu(x)=w(x)\,dx$. 

In the main topics of this note, we are concerned with the case that $w$ is a {\it Muckenhoupt $A_p$ weight} on $\R^{N+1}$, which satisfies for all balls $\BB\subset\R^{N+1}$, 
$$
\int_\BB w(x)\,dx\left(\int_\BB w(x)^{1/(1-p)}\,dx\right)^{p-1}\leq C|\BB|^p,
$$
where $|\BB|$ denotes the $(N+1)$-dimensional Lebesgue measure of $\BB$ and $C>0$ independent of $\BB$. We know that $A_p$ weights are $p$-admissible for the degenerate elliptic equations (\cite{HKM}, Chapter 20). 

Following \cite{HKM}, we denote the weighted Sobolev spaces by $H^{1,p}(E;\mu)$, defined as the completion of $C^\infty(E)$ under the norm
$$
\|h\|_{H^{1,p}(E;\mu)}=\left(\int_E|h|^p\,d\mu\right)^{1/p}+\left(\int_E|\nabla h|^p\,d\mu\right)^{1/p}.
$$
Denote by $H^{1,p}_0(E;\mu)$ the closure of $C^\infty_0(E)$ in $H^{1,p}(E;\mu)$.
Here and in the following contents, for simplicity we write $H^{1,p}(E;\mu)$ and $H^{1,p}_0(E;\mu)$ by $H^{1,p}(E)$ and $H^{1,p}_0(E)$ unless otherwise specified.

On $E\subset\R^{N+1}$ we consider the {\it weighted $p$-Laplace equation}
\begin{equation}\label{equ:gwequ}
    -\lap_{p,w}u=-\dv\left(w(x)|\nabla u|^p\nabla u\right)=0, 
\end{equation}
in which $w(x)$ is $p$-admissible and $1<p<\infty$.
The solutions of equations \ref{equ:gwequ} can be recognized as minimizers of the weighted $p$-Dirichlet integral
$$
\int_E|\nabla u|^p\,d\mu.
$$

A function $h: E\rightarrow\R$ is said to be weighted $p$-harmonic in $E$ if it is a continuous weak solution of \ref{equ:gwequ} in $E$. We denote
$$
\HH^p(E)=\{ h:h\ \text{is} \ \text{weighted}\ p\text{-harmonic}\ \text{in}\ E\}.
$$
By $H^p_E f$ we denote the solution of the weighted $p$-Dirichlet problem on $E$ with boundary data $f\in H^{1,p}(E)$.
We say that a lower semi-continuous function $u$ on $E$ is {\it $p$-superharmonic} in $E$ if $-\infty<u\leq+\infty$, $u$ is not identically $\infty$ in any component of $E$, and $H^p_{D}v\leq u$ in $D$ for every non-empty open set $D\Subset E$ and for all functions $h\in C(\overline{D})\cap \HH^p(D)$ 
the inequality $u\geq h$ on $\partial D$ implies $u\geq h$ in $D$. If $-u$ is $p$-superharmonic, then we say $u$ is $p$-subharmonic.

\begin{Def}
    Given a function $g$ on $\partial E$, we let $\U^p_g$ be the class of all $p$-superharmonic functions $u$ on $E$ bounded below such that $\liminf\limits_{E\ni x\rightarrow \xi}u(x)\geq g(\xi)$ for each $\xi\in\partial E$. The upper Perron solution of $g$ is defined by
    $$
    \overline{P}^p_Eg(x)=\inf_{u\in\U^p_g}u(x)\quad for \ x\in E.
    $$
    We define the lower Perron solution by
    $$
    \underline{P}^p_Eg(x)=\sup_{u\in\LL^p_g}u(x)\quad for\ x\in E,
    $$
    where $\LL^p_g=-\U^p_{-g}$ is the set of all $p$-subharmonic functions $u$ on $E$ bounded above such that $\limsup\limits_{E\ni x\rightarrow\xi}u(x)\leq g(\xi)$ for each $\xi\in\partial E$. If $\overline{P}^p_E g=\underline{P}^p_Eg$, we then say $g$ is {\it resolutive} and write $P^p_Eg$ for this function.
\end{Def}

Now we define $p$-harmonic measure based on Perron solutions.
\begin{Def}\label{def:whm}
   Let $\mu$ be defined as in formula \ref{def-mu}. Given an open set $F\subset\R^{N+1}$ and a Borel set $E\subset\partial F$, by the $p$-{\it harmonic measure} $\omega_p(E,F;\mu)$ we mean the upper Perron solution $\overline{P}^p_F\chi_E$ of the boundary function $\chi_E$ in $F$; see [\cite{AS}, Page 24], [\cite{HKM}, Page 201].
\end{Def}

We know that from the comparison principle, if $h\in\HH^p(E)$ is bounded such that $g(y)=\lim_{x\rightarrow y}h(x)$ exists for each $y\in\partial E$, then $\overline{P}^p_Eg=h=\underline{P}^p_Eg$; then the Perron solution agrees with the classical solution to the $p$-Dirichlet problem provided the latter exists. It's known that every (lower semi-) continuous function on $\partial E$ is resolutive. However, if $h$ is not bounded, then the existence of $g(y)=\lim_{x\rightarrow y}h(x)$ for $y\in\partial E$ is not enough to guarantee that $\overline{P}^p_Eg=h$; a classical counterexample we consider a singular $p$-harmonic function in the punctured ball $B(0,1)\setminus\{0\}$ if $\rcap_{p,\mu}\{0\}=0$.
We say that $\xi\in\partial E$ is {\it $p$-regular} if 
$$
\lim_{E\ni x\rightarrow\xi}P^p_Eg(x)=g(\xi)\quad {\rm for\ all}\ g\in C(\partial E).
$$
Then based on the statement above we get if $\xi\in\partial E$ is $p$-regular and $g$ is bounded on $\partial E$ and continuous at $\xi$, then
$$
\lim_{E\ni x\rightarrow\xi}\overline{P}^p_Eg(x)=\lim_{E\ni x\rightarrow\xi}\underline{P}^p_Eg(x)=g(\xi).
$$

\subsection{Fractional harmonic measure}\label{sec:fhm}
In this section, we use fractional Perron solutions and Caffarelli-Silvestre extension to give the representations of fractional harmonic measures respectively and show their equivalence in the case of $p=2$.  Of course, we can also use the classical Poisson kernel to define harmonic measure in fractional context as in \cite{Bogdan, Wu, CaSi} etc, but it's not our purpose in this note.

It's important to point out that, since our problem is non-local in nature, one does need to a {\it global} assumption in order to obtain such a result. In our case, the global assumption is hidden in the setting that either we consider the boundary data in the whole of $\R^N$ ($\R^{N+1}$) including points at $\infty$. We would see that these assumptions are necessary. One can also refer to similar settings in fractional obstacle problems in e.g. \cite{CSS, BFR}.

First, we define {\it fractional upper Perron solutions}.
\begin{Def}[\cite{KKP}, Definition 2]\label{def:fracperronsol}
    Let $\Omega$ be an open set, and let $0<s<1<p<\infty$. Assume that $g\in L^{p-1}_{sp}(\R^N)$. The upper class $\U^{s,p}_{g,\Omega}$ of $g$ consists of all functions $u$ such that
    \begin{itemize}
        \item[(i)] $u$ is $(s,p)$-superharmonic in $\Omega$,
        \item[(ii)] $u$ is bounded from below in $\Omega$,
        \item[(iii)] $\liminf\limits_{\Omega\ni y\rightarrow x}u(y)\geq$ ess$\limsup\limits_{\R^N\setminus\Omega\ni y\rightarrow x}g(y)$ for all $x\in\partial\Omega$,
        \item[(iv)] $u=g$ almost everywhere in $\R^N\setminus\Omega$.
    \end{itemize}
\end{Def}
The function $\overline{P}^{s,p}_\Omega g:=\inf\{u:u\in\U^{s,p}_{g,\Omega}\}$ is the {\it fractional-$p$ upper Perron solutions} with boundary datum $g$ in $\Omega$.

As in the weighted local framework, we construct the upper class on $\Omega$ of characteristic functions $\U^{s,p}_{\chi_G,\Omega}$ (denoted as $\U^{s,p}_{G,\Omega}$) to introduce the fractional harmonic measure.

\begin{Def}\label{def:fracindexp}
    Let $\Omega\subset\R^N$ be an open set. Let $G\subset\R^N\setminus\Omega$, such that, $G$ or its complement $G^c$ be a bounded set on $\R^N$. Let $\chi_G$ be the characteristic function of $G$. And let $0<s<1<p<\infty$. The {\it upper class} $\U^{s,p}_{G,\Omega}$ of $\chi_G$ consists of all functions $u$ such that
    \begin{itemize}
        \item[(i)] $u$ is $(s,p)$-superharmonic in $\Omega$,
        \item[(ii)] $u\geq 0$ on $\R^N$,
        \item[(iii)] $\liminf\limits_{\Omega\ni y\rightarrow x}u(y)\geq$ ess$\limsup\limits_{\R^N\setminus\Omega\ni y\rightarrow x}\chi_G(y)$ for all $x\in\partial\Omega$,
        \item[(iv)] $u=\chi_G$ almost everywhere in $\R^N\setminus\Omega$.
    \end{itemize}
\end{Def}
We recall that the points at infinity are always considered as part of the ``boundary" in non-local settings in the sense that if $G$ or $\R^N\setminus G$ is unbounded then we set $\chi_G(\infty)=1$ or $0$ (see Remark \ref{rem-boundedassump}). Also if $\Omega$ is unbounded, condition $(iii)$ is satisfied when $x=\infty$. However, through this note, we are concerned with bounded $\Omega$.

\begin{Def}\label{def:frachm}
    The function 
    $$
        \omega_p\left(G,\Omega;s\right)=\overline{P}^{s,p}_{\Omega}\chi_G=\inf\U^{s,p}_{G,\Omega},
    $$
    where the infimum is taken pointwise in $\R^N$.
\end{Def}

 It's easy to see this definition is well-defined by results in e.g. \cite{KKP}, since $\chi_G\in L^{\infty}(\R^N)\subset L^{p-1}_{sp}(\R^N)$. Also, it's obvious that $1\in\U^{s,p}_{G,\Omega}$, then we conclude the following result directly from Theorem $2$ in \cite{KKP}.
 \begin{Theorem}
     $\omega_p\left(G,\Omega;s\right)$ is $(s,p)$-harmonic in $\Omega$.
 \end{Theorem}

 \begin{Conv}
     When $p=2$, we just omit $p$. For simplicity, we write $\omega_2\left(G,\Omega;s\right)$ as $\omega(G,\Omega;s)$, and $\omega_2\left(G,\Omega;w\right)$ as $\omega(G,\Omega;w)$, where $w$ is an admissible weight defined in context.
 \end{Conv}

\quad

Now we turn to the definition of fractional harmonic measure through the approach of the Caffarelli-Silvestre extension.

Let $s\in(0,1)$ and $\theta=1-2s$. Denote $d\mu:=|t|^\theta\,dtdx$ with $x\in\R^N$ and $t\in\R$. 

In \cite{CS}, the authors established that the operator $(-\Delta)^s$ coincides with the Dirichlet to Neumann operator in the upper half space of $\R^{N+1}$. More precisely, given $u(x)$ defined in $\R^N$, extended it to $u^*(x,t)$ in $\R^{N+1}_+$. Then 
\begin{equation}\nonumber
    \left\{
    \begin{split}
         u^*(x,0)=u(x)&\quad\rm in\ \R^N,\\
         L_\theta u^*(x,t)=0&\quad\rm in\ \R^{N+1}_+,
    \end{split}
    \right.
\end{equation}
where $\R^{N+1}_+=\R^N\times(0,\infty)$ and 
$$
    L_\theta u^*(x,t):=-\dv_{x,t}\left(t^\theta\nabla_{x,t}u^*\right).
$$
This function $u^*$ can be obtained by minimizing the energy 
$$
    \min\left\{\int_{\R^{N+1}_+}t^\theta|\nabla_{x,t}v|^2\,dxdt:v(x,0)=u(x)\right\},
$$
and satisfies
$$
    \lim_{t\downarrow 0}t^\theta u^*(x,t)=(-\Delta)^su(x)\quad\rm in\ \R^N.
$$
By even reflection $u^*$ can be extended to $\R^{N+1}$, that is, $u^*(x,t)=u^*(x,-t)$, then  we have equivalent forms to Definition \ref{def:fracindexp} and Definition \ref{def:frachm} as elements of the class $\U^{|t|^\theta}_{G,\Omega^*}$ satisfying
\begin{equation}\label{equ-cs-global}
    \left\{
    \begin{split}
        &\text{(i)}\ L_\theta u^*(x,t)\geq 0\quad \text{in}\ \Omega^*,\\
        &\text{(ii)}\ u^*(x,t)\geq 0\quad \text{on}\ \R^{N+1},\\
        &\text{(iii)}\ \liminf\limits_{\Omega^*\ni (y,k)\rightarrow (x,0)}u^*(y,k)\geq\chi_{G}(x,0)\quad \text{for\ all}\ (x,0)\in\partial\Omega^*,\\
        &\text{(iv)}\ \lim\limits_{|(x,t)|\rightarrow\infty}u^*(x,t)=\chi_{G}(x,0),
    \end{split}
    \right.
\end{equation}
where $\Omega^*:=\R^{N+1}\setminus(\Omega^c\times \{0\})$, $\Omega^c:=\R^N\setminus\Omega$.

For any set $G\subset\R^N$ in this note, we don't distinguish $G$ and $G\times\{0\}$ unless otherwise specified.

\begin{Def}\label{def:frachmbyc-s}
    The function
    $$
        \omega(G,\Omega^*; |t|^\theta)=\overline{P}^{|t|^\theta}_{\Omega^*}\chi_{G}=\inf\U^{|t|^\theta}_{G,\Omega^*},
    $$
    where the infimum is taken pointwise in $\Omega^*\subset\R^{N+1}$, and where $\overline{P}^{|t|^\theta}_{\Omega^*}\chi_{G}$ is the upper Perron solution with weight $|t|^\theta$.
\end{Def}
Then obviously we know $\omega(G,\Omega^*; |t|^\theta)$ is harmonic in $\Omega^*$ from Chapter 11 in \cite{HKM}. 
 Then by Caffarelli-Silvestre equivalence, we have the following theorem.
\begin{Theorem}\label{thm-cs-frachm}
    $\omega\left(G,\Omega;s\right)=\omega(G,\Omega^*; |t|^\theta)$ on $\Omega\subset\R^N$.
\end{Theorem}

However, one of the great advantages of the Caffarelli-Silvestre extension lies in the fact that we can transform the non-local problem to be a localized problem proposed in a higher dimension space $\R^{N+1}$ in the sense of Lemma 4.1 in \cite{CS}. 
In the context below, we will use this feature frequently. For preparation, we need to define some abbreviated symbols. 
Now let $G,\Omega$ be any bounded set in $\R^N$ such that $G\subset\Omega^c$ and $\sup\limits_{x\in\Omega,y\in G}\rm dist(x,y)<\infty$. Let 
 $R$ be any real number larger than $\frac{\sup\limits_{x\in\Omega,y\in G}\rm dist(x,y)}{2}$. We write
 \begin{itemize}
     \item[] $B_R(x)$ be any ball in $\R^N$ with center $x\in\R^N$ and radius $R>0$;
     \item[] $\BB_R(x,0)$ be the extension ball of $B_R$ in $\R^{N+1}$ with same centre $(x,0)\in\R^{N+1}$ and same radius;
     \item[] $S_R(x)$ be the sphere of $B_R(x)$, and 
     \item[] $\s_R(x,0)$ be the sphere of $\BB_R(x,0)$.
 \end{itemize}

\begin{Remark}\label{rem-boundedassump}
    Throughout this note, we only consider the bounded case of $G$ or its complement $G^c$ on $\R^N$, since this concerns the way how to define boundary data in the extension approach. However, it's reasonable and interesting to consider the unbounded case of both $G$ and $G^c$, which is open.
\end{Remark}

\begin{Conv}\label{conv-setext}
     Let $\Omega$ be any set in $\R^N$. Then through this note, we write 
     $$
        \Omega^*=\R^{N+1}\setminus (\Omega\times\{0\}).
     $$

 \end{Conv}

\subsection{Weighted $p$-capacity in $\R^{N+1}$ and Besov capacity in $\R^N$}

A domain $E$ with no $p$-irregular boundary point is called a $p$-{\it regular domain}. And from the {\it Kellogg property}, we know the set of all $p$-irregular points on $\partial E$ is of $p$-Capacity zero. See Chapters 8 and 9 in \cite{HKM} for these accounts.

\begin{Proposition}[\cite{AS}, Proposition 2.1]\label{prop:rgset}
    Suppose $\|P^p_E\|_{\alpha\rightarrow\beta}<\infty$ for some $0<\beta\leq\alpha$. Then $E$ is a $p$-regular domain if and only if $\partial E$ has no $p$-trivial points.
\end{Proposition}
\begin{Remark}\label{rem:wsneq}
    As shown in section A.2 in \cite{BB} and \cite{Shanmug}, weighted Sobolev spaces with $p$-admissible weights in Euclidean setting are included in Newtonian spaces in metric spaces setting bearing doubling measure and $(1,p)$-Poincar\'e inequality, so in this note, we can directly use the results established in such Newtonian spaces. We will regularly use this property mainly in section \ref{sec-dgclass} concerning the H\"older regularity.
\end{Remark}
Then in light of Proposition \ref{prop:rgset} we get a $p$-regular domain by adding all $p$-trivial points. We may assume that $E$ concerned is $p$-regular in the sequel without any confusion.

Here we pause a little and insert the definition of weighted capacity, of which we borrow the settings in [\cite{HKM}, Chapter 2].

\begin{Def}\label{def:acap}
    Let $\BB\subset\R^{N+1}$ be a ball and $K\subset \BB$ be a compact set. The {\it weighted variational (condenser) capacity} of $K$ with respect to $\BB$ is
    $$
        \rcap_{p}(K,\BB)=\inf_v\int_\BB|\nabla v|^p\,d\mu,
    $$
    where the infimum is taken over all $v\in C^\infty_0(\BB)$ such that $v\geq 1$ on $K$.
\end{Def}

While for all Borel set $E\subset \BB$, we can define
$$
    \rcap_p(E,\BB)=\sup_{\rm compact\ K\subset E}\rcap_p(K,\BB).
$$
A set $E\subset\R^{N+1}$ is of zero $p$-Capacity if for all ball $\BB\subset\R^{N+1}$, there holds
$$
\rcap_p(E\cap \BB, \BB)=0.
$$

Then based on the Definitions \ref{def:acap}, we can directly define the $|t|^\theta$-Muckenhoupt capacity just by letting $w(x)$ be the weights $|t|^\theta$ with $\theta$ to be defined later. So let $\BB\subset\R^{N+1}$ be a ball and $K\subset\BB$ be a compact set. We denote the $|t|^\theta$-Muckenhoup weighted capacity of $K$ with respect to $\BB$ by
$$
    \rcap_{p,|t|^\theta}(K,\BB)=\inf_v\int_\BB|\nabla v|^p\,d\mu_\theta,\quad\text{with\ }\,d\mu_\theta=|t|^\theta\,dtdx,\  t\in\R,\  x\in\R^N,
$$
where the infimum is taken over all $v\in C^\infty_0(\BB)$ such that $v\geq 1$ on $K$ and $v=0$ q.e. on $\R^{N+1}\setminus\BB$.

As we aim to formulate the boundary regularity of fractional Laplacians $(-\Delta)^s$ in $\R^N$ via Caffarelli-Silvestre extension, we need the capacity associated with the Besov spaces $B^s_p(\R^N)$. Following Maz'ya \cite{Mazya3} (p.512) we define the Besov semi-norm by
$$
\|v\|_{B^s_p(\R^N)}:=\left(\int_{\R^N}\int_{\R^N}\frac{|v(x)-v(y)|^p}{|x-y|^{N+sp}}\,dxdy\right)^{1/p}, \quad 0<s<1<p,
$$
and Theorem 1 in \cite{Mazya3} (p.512) asserts that for all $v\in C^\infty_0(\R^N)$,
\begin{equation}\label{global-besov}
    \|v\|_{B^s_p(\R^N)}\simeq\inf_v\left\||t|^{1-s-1/p}\nabla\Tilde{v}\right\|_{L^p(\R^{N+1})},
\end{equation}
where the infimum is taken over all extensions $\Tilde{v}\in C^\infty_0(\R^{N+1})$ of $v$.

\begin{Def}\label{def-besovcap}
    Let $B\subset\R^N$ be a ball and $K\subset B$ be a compact set. The {\it variational (condenser) Besov capacity} of $K$ with respect to $B$ is
    $$
    \rcap_{B^s_p}(K,B)=\inf_v\|v\|^p_{B^s_p(\R^N)},
    $$
    where the infimum is taken over all $v\in C^\infty_0(B)$ such that $v\geq 1$ on $K$. For a Borel set $\Omega\subset B$, we let
    $$
    \rcap_{B^s_p}(\Omega,B)=\sup_{\rm compact\ K\subset\Omega}\rcap_{B^s_p}(K,B).
    $$
    We define a set $\Omega\subset\R^{N}$ of zero $B^s_p$-capacity if for all ball $B\subset\R^{N}$, there holds
    $$
        \rcapb(\Omega\cap B, B)=0.
    $$
\end{Def}

Here and in what follows, we use the notation $X\simeq Y$ if there is a positive constant $C$ independent of $X$ and $Y$ such that $X/C\leq Y\leq CX$. We define the notations of one-sided inequalities $\lesssim$ and $\gtrsim$ in a similar way.

The following lemma relates the weighted capacity $\rcapw$ in $\R^{N+1}$ to a Besov capacity $\rcap_{B^s_p}$ in $\R^N$. Here we identify $K\subset\R^N$ with $K\times\{0\}\subset\R^{N+1}$. For the proof details one can refer to \cite{Bjorn, Mazya3}.
\begin{Lemma}[\cite{Bjorn}, Lemma 2.2]\label{lem-besovcap}
    Let $z_0=(x_0,0)\in\R^{N+1}$ and $K\subset\overline{B_r(x_0)}\subset\R^N$ be a compact set. Then
    $$
    \rcapw(K,B_{2r}(z_0))\simeq\inf_v\|v\|^p_{B^s_p(\R^N)},\quad where\ s=1-\frac{1}{p}-\frac{\theta}{p},
    $$
    and the infimum is taken over all $v\in C^\infty_0(B(x_0,2r))$ such that $v\geq 1$ on $K$.
\end{Lemma}
Then we conclude the following equivalence lemma.
\begin{Lemma}[\cite{Bjorn}, Lemma 1.3]\label{lem-equivalentcap}
    Let $\Omega\subset\overline{B_r(x_0)}\subset\R^N$ be a Borel set and $z_0=(x_0,0)\in\R^{N+1}$. Then for $1<p<\infty$ and $s\in(0,1)$ we have
    $$
        \rcapb\left(\Omega, B_{2r}(x_0)\right)\simeq\rcapw\left(\Omega\times\{0\}, \BB_{2r}(z_0)\right),\quad\text{where}\ \theta=p(1-s)-1.
    $$
\end{Lemma}
We also give the following estimated results for Besov capacities.
\begin{Lemma}[\cite{Bjorn}, Lemma 2.4]\label{lem-besovcapest}
    Let $0<\rho\leq r$, $x_0\in\R^N$, $z_0=(x_0,0)$, and $s=1-1/p-1/a$, then
    \[
    \begin{split}
        \rcapw\left(\BB_{\frac{1}{2}\rho}(z_0), \BB_{2r}(z_0)\right)
        &\lesssim\rcapb\left(\overline{B_\rho(x_0)},B_{2r}(x_0)\right)\\
        &\lesssim\rcapw\left(\BB_{\rho}(z_0), \BB_{2r}(z_0)\right).
    \end{split}
    \]
    In particular, $\rcapb\left(\overline{B_{cr}(x_0)},B_{2r}(x_0)\right)\simeq r^{N-sp}$ with comparison constant depending only on $N$, $s$, $p$, and $c\in(0,1]$.
\end{Lemma}

For some geometrical characterization of higher boundary regularity, we need information on measuring the boundary itself.

\begin{Def}\label{def:capdensity}
    Set $G$ satisfies the {\it uniformly fractional $p$-fat} or the {\it fractional $p$-capacity density condition} if there exist constants $C_0>0$ and $r_0>0$ such that
    \begin{equation}\label{equ:fat}
        \frac{\rcapb(G\cap B_r(a), B_{2r}(a))}{\rcapb(B_r(a), B_{2r}(a))}\geq C_0
    \end{equation}
    whenever $a\in G$ and $0<r<r_0$.
\end{Def}
See \cite{Lewis} for more details on uniform fatness in the Euclidean setting, and \cite{BMS} on metric spaces.

\begin{Remark}
    In \cite{LZLH} the authors gave a sufficient condition for boundary H\"older regularity via volume density condition for elliptic equations and fractional Laplacians, i.e., once the boundary point $a\in\partial G$ satisfies
$$
    \frac{\mu(G\cap B_r(a))}{\mu(B_r(a))}\geq C>0, \quad for\ r_0>r>0,
$$
there holds the H\"older regularity up to point $a$ for harmonic functions in $E$. We can see the volume density condition satisfies the capacity density condition (see Proposition \ref{prop-lebeguetocap}), but the inverse is not the case. 
\end{Remark}

\begin{Proposition}\label{prop-lebeguetocap}
    Let $\R^N\setminus\Omega$ satisfy the volume density condition, then it is fractional uniformly fat, and hence $\|P_\Omega\|_{\alpha\rightarrow\alpha}<\infty$ for some $\alpha\in(0,\sigma_0]$, where $\sigma_0$ is defined in Definition \ref{def-sigma0}.
\end{Proposition}
We give the proof of this proposition at the end of this note.

\medskip

However, before going on, we should exclude some singular phenomena. We say that a property holds q.e. (quasi-everywhere) if it holds outside a polar set, whose capacity is zero.
As we all know, in a punctured ball $\Tilde{B}=B_1(0)\setminus\{0\}$, the original point is irregular, and yet $\|P_{\Tilde{B}}\|_{\alpha}<\infty$ for all $0<\alpha<1$. There is the isolated boundary point that triggered this strange phenomenon. That's why we should exclude these (essentially) isolated points.

Now let's go back to our original motivation, i.e., whether better continuity of $g$ on the complement of $\Omega$ results in a better regularity of $P_\Omega g$, also, whether there is a similar equivalence between the finiteness of $\|P_\Omega g\|_{\alpha\rightarrow\beta}$ (defined as in \ref{op:fracmap}, here $0<\alpha\leq\beta\leq 1$) and regularity of $\partial\Omega$ as in the sense of classical Wiener criterion. In fact, by \cite{Aikawa, AS} this is not always true. Then to avoid such pathological cases we need to introduce {\it fractional trivial boundary point}, i.e., for $a\in\partial\Omega$ being a  fractional trivial boundary point of $\Omega$, if there is $r>0$ such that $\rcapb\left(\partial\Omega\cap B_r(a)\right)=0$.

\begin{Def}
    The point $a\in\partial\Omega$ is a {\it trivial boundary point} if there is $r>0$ such that $B_r(a)\setminus\Omega$ is polar.
\end{Def}
We can easily see that the original point in the example above is a trivial boundary point. We have the following description of the trivial boundary point in the same sense as Proposition 1 in \cite{Aikawa}. We omit the proof, which is just by the definition of regular and irregular points.
\begin{Proposition}\label{prop-trivialbdpoint}
    Let $G$ be the set of all trivial boundary points of $\Omega$. Then
    \begin{itemize}
        \item[(i)] $G$ is polar;
        \item[(ii)] $\Tilde{\Omega}=\Omega\cup G$ is a domain without trivial boundary points;
        \item[(iii)] $\Tilde{\Omega}=\{x\in\R^N: \text{there is $r>0$ such that $B_r(x)\setminus\Omega$ is polar}\}$;
        \item[(iv)] $\partial\Tilde{\Omega}\subset\partial\Omega$;
        \item[(v)] $P_\Omega f=P_{\Tilde{\Omega}}\Tilde{f}$ on $\Omega$, where $\Tilde{f}=f|_{\R^N\setminus\Tilde{\Omega}}$.
    \end{itemize}
\end{Proposition}

Then we have the following result.
\begin{Theorem}\label{thm-regtrivialbdr}
    Let $\|P_\Omega\|_\alpha<\infty$ for some $\alpha>0$. Then $\Omega$ is regular if and only if $\Omega$ has no trivial boundary point.
\end{Theorem}
\begin{proof}
    Obviously, we have that if $\Omega$ has a trivial point, then $\Omega$ is irregular.

    Conversely, if $\Omega$ is irregular, we suppose that there are no trivial points on $\partial\Omega$. Then for an arbitrary point $a\in\partial\Omega$, we set $u=P_\Omega\phi_{a,\alpha}$, where $\phi_{a,\alpha}$ is defined in \ref{funcdef-phi-bdr}. By the assumption, we would have
    \begin{equation}\label{equ-trivial-1}
        \lim_{\Omega\ni x\rightarrow b}u(x)=\phi_{a,\alpha}(b)\quad \text{for $b\in\partial\Omega$}.
    \end{equation}
    Now let $b\in\partial\Omega$ and $r>0$. Then by assumption, we have that $u\in\Lambda_\alpha(\Omega)$.
    Hence we have
    $$
        |u(x)-u(y)|\leq \CC r^\alpha\quad \text{for $x,y\in B_r(b)\cap\Omega$}.
    $$
    As we know $b$ is non-trivial, by Kellog property we can find some regular point $b^\prime\in\Omega\cap B_r(b)$.
    Then let $y\rightarrow b^\prime$, and it yields that $|u(x)-\phi_{a,\alpha}(b^\prime)|\leq\CC r^\alpha$. Hence by the definition of $\phi_{a,\alpha}$, we obtain
    \[
    \begin{split}
        |u(x)-\phi_{a,\alpha}(b)|&=|u(x)-\phi_{a,\alpha}(b^\prime)+\phi_{a,\alpha}(b^\prime)-\phi_{a,\alpha}(b)|\\
        &\leq\CC r^\alpha+d(b,b^\prime)^\alpha\\
        &\leq\CC r^\alpha+(2r)^\alpha\quad\text{for $x\in B_r(b)$},
    \end{split}
    \]
    which yields \ref{equ-trivial-1} by letting $r\rightarrow 0$.

    As $\phi_{a,\alpha}(a)=0$ and $\phi_{a,\alpha}(b)>0$ for $\R^N\setminus\Omega\ni\forall b\neq a$ by the definition of $\phi_{a,\alpha}$, we get that $u$ is a barrier function at $a$ by \ref{equ-trivial-1}. Hence $a$ is a regular boundary point (see \cite{LL}). Hence $\Omega$ is regular by the arbitrariness of $a\in\partial\Omega$.
\end{proof}

Based on the analysis above, we know a fractional trivial boundary point can be regarded as an interior point from the point of view of potential theory. Adding all these trivial boundary points to the domain, we obtain a domain with no trivial boundary point; the potential theoretical property of the resulting domain is the same as that of the original domain. In light of Proposition \ref{prop-trivialbdpoint} and Theorem \ref{thm-regtrivialbdr},
we may assume that $\Omega\subset\R^N$ is fractional regular in the sequel.

\subsection{Main results}

Let $\Omega\subset\R^N$ with dimension $N\geq 2$. Let operator $P_\Omega$ be the operator defined in \ref{op:fracmap}. We define the boundary test function with H\"older regularity as 
    \begin{equation}\label{funcdef-phi-bdr}
        \phi_{a,\sigma}(x):=\min\{d(x,a)^\sigma,1\}\quad{\rm for}\ a,x\in\R^N\setminus\Omega,
    \end{equation}
    which is bounded and continuous on $\R^N\setminus\Omega$, and so that $P_\Omega\phi_{a,\sigma}$ is well-defined if $\R^N\setminus\Omega$ is regular.

We also need the global and local decay properties for the fractional harmonic measure.
\begin{Def}\label{def-gd}
    Fractional harmonic measure satisfies the {\it Global harmonic measure decay} for some $\alpha>0$ ($\GHMDs_\alpha$), if there exist constants $C_2\geq 1$ and $0<r_0<\diam{\Omega}/2$ such that whenever $a\in\partial\Omega$ and $0<r<r_0$, for every $x\in\Omega\cap B_r(a)$ there holds
        $$
            \omega\left(\Omega^c\setminus B_r(a), \Omega;s\right)(x)\leq C_2\left(\frac{d(x,a)}{r}\right)^\alpha,
        $$
        where $\Omega^c$ denotes the complement of $\Omega$ in $\R^N$.
\end{Def}

\begin{Def}\label{def-ld}
    Fractional harmonic measure satisfies the {\it Local harmonic measure decay} for some $\alpha>0$ ($\LHMDs_\alpha$), if there exist constants $C_3\geq 1$ and $0<r_0<\diam{\Omega}/2$ such that whenever $a\in\partial\Omega$ and $0<r<r_0$, for every $x\in\Omega\cap B_r(a)$ there holds
        $$
            \omega\left((B_r(a))^c, \Omega\cap B_r(a);s\right)(x)\leq C_3\left(\frac{d(x,a)}{r}\right)^\alpha,
        $$
        where $(B_r(a))^c$ denotes the complement of $B_r(a)$ in $\R^N$.
\end{Def}

As we would see in the proof, we would utilize the exponent of interior H\"older regularity in the De Giorgi class (denoted as $\alpha_0<1$, see Lemma \ref{lem-holderdg}) for $p$-Lalacian functions as shown in \cite{AS, KS}. Also as shown in \cite{BLS}, the authors established a higher and almost critical interior H\"older regularity (denoted as $s<\sigma_1:=\min\{\frac{sp}{p-1}, 1\}$ with $p\geq 2$) locally in $\Omega$ for fractional $p$-harmonic functions, we are convinced that the boundary H\"older regularity can not exceed $\sigma_1$. In the meanwhile, Ros-Oton and Serra have established almost critical boundary regularity for fractional Laplacians in \cite{RS, RS1, RS2}. One can also see in \cite{ILPS, IMS} a global regularity via barrier analysis, in which we can find the H\"older exponent is always no larger than $s$. So to be on the safe side, we give the following definition for boundary H\"older exponents. 
\begin{Def}\label{def-sigma0}
    Define $\sigma_0:=\min\{s, \alpha_0\}$.
\end{Def}

\begin{Theorem}\label{thm:main1}
     Let $s\in(0,1)$. Let $\Omega\subset\R^{N}$ be a bounded $s$-regular domain. Suppose $0<\sigma\leq\sigma_0$, where $\sigma_0$ is a positive constant defined in Definition \ref{def-sigma0}. Consider the following four conditions:
    \begin{itemize}
        \item[(i)]$\|P_\Omega\|_{\sigma\rightarrow\sigma}<\infty$.
        \item[(ii)] There exists a constant $C_1$ such that whenever $a\in\partial\Omega$ and for every $x\in\Omega$,
        $$
            P_\Omega\phi_{a,\sigma}(x)\leq C_1d(x,a)^\sigma.
        $$
        \item[(iii)] $GHMD^s_\sigma$ holds on $\Omega$.
        \item[(iv)]  $LHMD^s_\sigma$ holds on $\Omega$.
    \end{itemize}
    Then we have 
    $$
        (i)\Longleftrightarrow(ii)\Longrightarrow(iii)\Longleftrightarrow(iv).
    $$
    If $(iii)$ or, equivalent, $(iv)$ holds with some $\sigma^\prime>\sigma$, then $(i)$ and, equivalently, $(ii)$ hold.
\end{Theorem}

If we ignore the exact H\"older exponent $\sigma$, we have the following characterization.
\begin{Theorem}\label{thm:main2}
    Let $s\in(0,1)$. Let $\Omega$ be a bounded $s$-regular domain. Then the following five statements are equivalent:
    \begin{itemize}
        \item[(i)]$\|P_\Omega\|_{\sigma\rightarrow\sigma}<\infty$ for some $\sigma>0$.
        \item[(ii)] There holds for some $\sigma>0$
        $$ 
            P_\Omega\phi_{a,\sigma}(x)\leq C_1d(x,a)^\sigma.
        $$
        \item[(iii)] $GHMD^s_\sigma$ holds for some $\sigma>0$.
        \item[(iv)]  $LHMD^s_\sigma$ holds for some $\sigma>0$.
        \item[(v)]  $\R^{N}\setminus\Omega$ is uniformly $s$-fat.
    \end{itemize}
\end{Theorem}

\section{Minimizers and De Giorgi class}\label{sec-dgclass}

The contents in this section are standard. We mainly refer to the results in \cite{HKM, KS, AS}. And \cite{KS, AS} are concerned with more general metric spaces; however, we can utilize their results directly, since weighted Sobolev spaces equipped with Muckenhoupt $A_p$-weights ( even with $p$-admissible weights)  in Euclidean setting are included in Newtonian spaces in metric spaces bearing with doubling measures and so-called $(1,p)$-Poinc\'are inequality (see section A.2 in \cite{BB}, \cite{Shanmug}). 

In this section, we just state the results without proof.
\begin{Def}
    We call a function $u$ on $\Omega$ a $p$-{\it minimizer} in $\Omega$ if $u\in H^{1,p}_{\rm loc}(\Omega)$ and 
    \begin{equation}\label{def-minimizer}
        \int_K|\nabla u|^p \,d\mu\leq \int_K|\nabla (u+\varphi)|^p \,d\mu
    \end{equation}
    for all relatively compact subsets $K$ of $\Omega$ and for every function $\varphi\in H^{1,p}_0(K)$. A   $p$-{\it harmonic} function is a continuous   $p$-minimizer.
\end{Def}

\begin{Def}
    We call a function $u$ a   $p$-{\it superminimizer} in $\Omega$ if $u\in H^{1,p}_{\rm loc}(\Omega)$ and the energy minimizing inequality \ref{def-minimizer} holds for all relatively compact subsets $K$ in $\Omega$ and for every non-negative function $\varphi\in H^{1,p}_0(K)$.
\end{Def}
\begin{Remark}\label{rem-har-mini}
    Let $u$ be a   $p$-superminimizer in $\Omega$. Then the lower regularization ${\rm ess}\liminf_{y\rightarrow x}u(y)$ is a lower semicontinuous representative and it is a   $p$-superharmonic function. Conversely, a bounded   $p$-superharmonic function (resp.   $p$-subharmonic) function is a   $p$-superminimizer (resp. $p$-subminimizer). An unbounded $p$-superharmonic function need not be a $p$-superminimizer; however, the truncation of such a $p$-superharmonic function is a $p$-superminimizer.
\end{Remark}
\begin{Def}
    Given an open set $\Omega$, a function $u\in H^{1,p}_{\rm loc}(\Omega)$ is said to belong to the {\it De Giorgi class} $\rm DG_p(\Omega)$ if there are constants $C>0$ such that
    $$
        \int_{B_\varrho(z)}|\nabla (u-k)_+|^p\,d\mu\leq\frac{C}{(\rho-\varrho)^p}\int_{B_\rho(z)}(u-k)^p_+\,d\mu
    $$
    whenever $k\in\R$, $0<\varrho<\rho<\infty$, and $B_\rho(z)\subset\Omega$.
\end{Def}

\begin{Lemma}[\cite{AS}, Lemma 3.2]\label{lem-subminidg}
    If $u$ is a quasi-subminimizer on $\Omega$, then $u\in DG_p(\Omega)$. If $u$ is a quasi-minimizer on $\Omega$, then both $u$ and $-u$ belong to $DG_p(\Omega)$.
\end{Lemma}

We denote by ${\rm osc}_E$ the oscillation $\sup_Eu-\inf_Eu$.

\begin{Lemma}[\cite{KS}, Lemma 3.4]\label{lem-holderdg}
    Suppose that both $u$ and $-u$ belong to $DG_p\left(B_{2\rho}(x)\right)$. Then 
    $$
        {\rm osc}_{B_\varrho(x)}u\leq C\left(\frac{\varrho}{\rho}\right)^{\alpha_0}{\rm osc}_{B_\rho(x)}u\quad{\rm for}\ 0<\varrho\leq\rho
    $$
    for some $0<\alpha_0\leq 1$ and $C\geq 1$ independent of $u$, $x$ and $\rho$.
\end{Lemma}

\begin{Lemma}[\cite{KS}, Lemma 3.6]\label{lem-decaydg}
    Let $0<\rho_1<\infty$ and $u\in DG_p\left(B_{2\rho_1}(x)\right)$. Suppose $0\leq u\leq 1$ on $B_{2\rho_1}(x)$ and 
    $$
        \frac{\mu\left(\left\{x\in B_{\rho_1}(x):u>1-s\right\}\right)}{\mu\left(B_{\rho_1}(x)\right)}\leq\gamma<1
    $$
    for some $0<s<1$. Then there exists $t=t(p, \gamma, s)>0$ such that
    $$
        u\leq 1-t\quad{\rm on}\ B_{\rho_1/2}(x).
    $$
\end{Lemma}

\begin{Corollary}[\cite{KS}, Corollary 3.7]\label{col-chaindecay}
    Let $o<\rho<\infty$ and $B_{\rho/2}(x_1)\cap B_{\rho/2}(x_2)\neq\emptyset$. Suppose $u\in DG_p\left(B_{2\rho}(x_2)\right)$ with $0\leq u\leq 1$ in $B_{2\rho}(x_2)$. If $u\leq 1-\varepsilon_1$ on $B_{\rho/2}(x_1)$ for some $\varepsilon_1>0$, then there exists a positive constant $\varepsilon_2=\varepsilon_2(\varepsilon_1)<1$ such that $u\leq 1-\varepsilon_2$ on $B_{\rho/2}(x_2)$.
\end{Corollary}

\section{Proof of Theorem \ref{thm:main1}}

 The proof includes a series of lemmas. 

\begin{Lemma}
    Condition (i) $\Leftrightarrow$ Condition (ii).
\end{Lemma}
\begin{proof}
    First, we prove condition (i) $\Rightarrow$ condition (ii). From the definition we know $\phi_{a,\sigma}\in\Lambda_{\sigma}(\R^N\setminus\Omega)$ with $\|\phi_{a,\sigma}\|_{\Lambda_\sigma(\R^N\setminus\Omega)}\leq 2$, which yields
    \[
    \begin{split}
      |P_\Omega\phi_{a,\sigma}(x)-P_\Omega\phi_{a,\sigma}(y)|&\leq \|P_\Omega\|_{\sigma\rightarrow\sigma}|\phi_{a,\sigma}(x)-\phi_{a,\sigma}(y)|\\
      &\leq 2\|P_\Omega\|_{\sigma\rightarrow\sigma}d(x,y)^\sigma\quad\rm for\ x,y\in\Omega.  
    \end{split}
    \]
    Since $a$ is an $s$-regular point by assumption, we conclude that condition (ii) with $C_1=2\|P_\Omega\|_{\sigma\rightarrow\sigma}d$ by letting $y\rightarrow a$.

    Now we prove that condition (ii) $\Rightarrow$ condition (i).

    As in the extension setting of Definition \ref{def:frachmbyc-s}, we have $\partial\Omega^*=\Omega^c\times\{0\}$, where $\Omega^c=\R^N\setminus\Omega$. Let $f\in \Lambda_\sigma(\Omega^c)$, then we know $f\in \Lambda_\sigma(\partial\Omega^*)$. As we know from Caffarelli-Silvestre extension that $P_{\Omega^*}f=P_\Omega f$ on $\Omega$, then by the maximum principle (see \cite{HKM}, 6.5) we have 
    $$
        \sup_{x\in\Omega}|P_{\Omega}f(x)|\leq\sup_{x\in\Omega^*}|P_{\Omega^*}f(x)|\leq\sup_{\xi\in\partial\Omega^*}|f(\xi)|\leq\|f\|_{\Lambda_\sigma(\partial\Omega^*)}=\|f\|_{\Lambda_\sigma(\R^N\setminus\Omega)}.
    $$ 
    Since $\Omega$ is bounded, we just need to show for $\forall x,y\in\Omega$ such that $d(x,y)\leq 1$, there holds
    $$
        |P_\Omega f(x)-P_\Omega f(y)|\leq C \|f\|_{\Lambda_\sigma(\R^N\setminus\Omega)}d(x,y)^\sigma.
    $$
    
    Now let $x,y\in\Omega$ with $d(x,y)\leq 1$. Let $M=\max\{d(x,\partial\Omega), d(y,\partial\Omega)\}/2$. Generally, we can just suppose $d(x,\partial\Omega)\geq d(y,\partial\Omega)$. Since $\partial\Omega$ is compact, there $\exists x_0\in\partial\Omega$ such that $d(x,\partial\Omega)=d(x,x_0)$.
    Set $f_0:\R^N\setminus\Omega\mapsto\R$ by $f_0(\xi):=f(\xi)-f(x_0)$ for $\forall\xi\in\R^N\setminus\Omega$. Then we have the inequality 
    \begin{equation}\nonumber
        |f_0(\xi)|\leq\left\{
        \begin{split}
            &|f(\xi)-f(x_0)|\leq\|f\|_{\Lambda_\sigma(\R^N\setminus\Omega)}d(\xi,x_0)^\sigma\leq\|f\|_{\Lambda_\sigma(\R^N\setminus\Omega)}\phi_{x_0,\sigma}(\xi),\quad d(x_0,\xi)\leq 1;\\
            &|f(\xi)|+|f(x_0)|\leq2\|f\|_{\Lambda_\sigma(\R^N\setminus\Omega)}\leq2\|f\|_{\Lambda_\sigma(\R^N\setminus\Omega)}\phi_{x_0,\sigma}(\xi),\quad d(x_0,\xi)\geq 1,
        \end{split}
        \right.
    \end{equation}
    where $\phi_{x_0,\sigma}$ is the boundary test function.
    Now from $(ii)$ we know for $\forall z\in\Omega$ there holds
    \begin{equation}\label{equ-main-1}
        |P_{\Omega}f_0(z)|=|P_{\Omega^*}f_0(z)|\leq 2C_1\|f\|_{\Lambda_\sigma(\R^N\setminus\Omega)}d(z,x_0)^\sigma
    \end{equation}
    by the definiton of $\phi_{x_0,\sigma}$.

    Then let $x,y\in\Omega$, we have two cases as:
    \begin{itemize}
        \item[\underline{Case 1}:] $d(x,y)\leq M=d(x,\partial\Omega)/2$.
        Since $P_{\Omega^*}f_0$ is $|t|^\theta$-harmonic, we have that $P_{\Omega^*}f_0$ belongs to De Giorgi class $DG^{|t|^\theta}(\BB_{2M}(x))$. Hence from Lemma \ref{lem-holderdg} we obtain
        $$
            \mathop{\rm osc}\limits_{B_r(x)}P_{\Omega^*}f_0\leq\mathop{\rm osc}\limits_{\BB_r(x)}P_{\Omega^*}f_0
            \leq C\left(\frac{r}{M}\right)^{\alpha_0}\mathop{\rm osc}\limits_{\BB_M(x)}P_{\Omega^*}f_0,\quad  {\rm for}\  0<r\leq M.
        $$
        Then from the definition of $M$ we observe that 
        $$
            d(z,x_0)\leq d(x,z)+d(z,x_0)\leq 3M
        $$
        whenever $z\in \BB_M(x)$. Then by \ref{equ-main-1} we have
        $$
            \mathop{\rm osc}\limits_{B_r(x)}P_{\Omega^*}f_0\leq C\|f\|_{\Lambda_\sigma(\Omega^c)}M^\sigma.
        $$
        Hence there holds
        \[
        \begin{split}
                |P_{\Omega^*}f(x)-P_{\Omega^*}f(y)|&=|P_{\Omega^*}f_0(x)-P_{\Omega^*}f_0(y)|\\
                &\leq C\left(\frac{d(x,y)}{M}\right)^{\alpha_0}\|f\|_{\Lambda_\sigma(\R^N\setminus\Omega)}M^\sigma\\
                &\leq C\|f\|_{\Lambda_\sigma(\R^N\setminus\Omega)}d(x,y)^\sigma,
        \end{split}
        \]
        where in the last inequality we used $\sigma\leq\alpha_0$ and $d(x,y)\leq M$.
        
        \item[\underline{Case 2}:] $d(x,y)\geq M=d(x,\partial\Omega)/2$ implies $d(y,x_0)\leq d(x,y)+d(x,x_0)\leq 3d(x,y)$. Then again from \ref{equ-main-1} we obtain
        \[
        \begin{split}
            |P_{\Omega^*}f(x)-P_{\Omega^*}f(y)|&=|P_{\Omega^*}f_0(x)-P_{\Omega^*}f_0(y)|
            \leq |P_{\Omega^*}f_0(x)|+|P_{\Omega^*}f_0(y)|\\
            &\leq 2C_1\|f\|_{\Lambda_\sigma(\R^N\setminus\Omega)}\left(d(x,x_0)^\sigma+d(y,x_0)^\sigma\right)\\
            &\leq 2C_1\|f\|_{\Lambda_\sigma(\R^N\setminus\Omega)}\left(2^\sigma+3^\sigma\right)d(x,y)^\sigma.
        \end{split}
        \]
    \end{itemize}
    By combining two cases we conclude the desired results by noticing \ref{equ-main-1}.

\end{proof}

\begin{Lemma}
    Theorem \ref{thm:main1} (ii)$\Rightarrow$ Theorem \ref{thm:main1} (iii).
\end{Lemma}

\begin{proof}
    Let $a\in\partial\Omega$. Let $0<r<\min\{r_0, \frac{1}{2}\rm diam(\Omega)\}$, where $r_0$ is the same as in the definition of GHMD. Then we have
    $$
      \chi_{\R^N\setminus\left(\Omega\cup B_r(a)\right)}(\xi)\leq r^{-\sigma}\phi_{a,\sigma}(\xi),\quad\rm for\ \xi\in\R^N\setminus\Omega.
    $$
    Hence, the comparison principle condition implies that
    \[
        \begin{split}
            \omega\left(\R^N\setminus\left(\Omega\cup B_r(a)\right),\Omega^*;|t|^\theta\right)&=\overline{P}_{\Omega^*}\chi_{\R^N\setminus\left(\Omega\cup B_r(a)\right)}
            \leq r^{-\sigma}\overline{P}_{\Omega^*}\phi_{a,\sigma}
        \end{split}    
    \]
    Then by condition (ii) and the definition of fractional harmonic measure, we obtain
    \[
    \begin{split}
        \omega\left(\R^N\setminus\left(\Omega\cup B_r(a)\right),\Omega;s\right)(x)&=\omega\left(\R^N\setminus\left(\Omega\cup B_r(a)\right),\Omega^*;|t|^\theta\right)(x)\\
        &\leq C_2r^{-\sigma}|x-a|^\sigma\quad\text{for\ $\forall x\in\Omega\cap B_r(a)$},
    \end{split}
    \]
    which concludes condition (iii).
\end{proof}

\begin{Lemma}
    ${\rm LHMD}^s_{\sigma^\prime}$ $\Rightarrow$ condition $\rm (ii)$ with $\sigma^\prime>\sigma$.
\end{Lemma}
\begin{proof}
    In this proof, we mainly use linearity of $-\dv\left(|t|^\theta\nabla u(x,t)\right)$ and comparison principle.
    
    Let $a\in\partial\Omega$, $x\in\Omega$, such that $r=d(x,a)$. Then based on the definition of $\phi_{a,\sigma}$ and comparison principle we have
    \[
    \begin{split}
        P_{\Omega^*}\phi_{a,\sigma}
        &\leq P_{\Omega^*}\left(r^\sigma\chi_{\partial\Omega^*}\right)+\Sigma^\infty_{j=1}(2^jr)^\sigma P_{\Omega^*}\chi_{\partial\Omega^*\cap A^*_a(2^{j-1}r,2^jr)}\\
        &\leq r^\sigma+\Sigma^\infty_{j=1}(2^jr)^\sigma\omega(\cdot, \partial\Omega^*\cap \BB_a(2^{j-1}r),\Omega^*)\quad \rm on\ \Omega.\\
        \overset{\rm(iv)}{\Rightarrow}P_{\Omega^*}\phi_{a,\sigma}(x)&\leq r^\sigma+\Sigma^\infty_{j=1}(2^jr)^\sigma C_3\left(\frac{|x-a|}{|2^jr|}\right)^{\sigma^\prime}=\left(1+\frac{2^\sigma C_3}{1-2^{\sigma^\prime-\sigma}}\right)|x-a|^\sigma.
    \end{split}
    \]

\end{proof}

\section{Proof of Theorem \ref{thm:main1} continued}

Now we prove $\GHMDs_\sigma\Leftrightarrow \LHMDs_\sigma$.

\begin{Lemma}
    $\LHMDs_\sigma\Rightarrow\GHMDs_\sigma$.
\end{Lemma}

\begin{proof}
    Let $a\in\partial\Omega$ and $0<r<r_0$, where $r_0$ is specified in definition of $\LHMDs_\sigma$. By noticing that
    $$
        \chi_{\R^N\setminus B_r(a)}\geq \omega\left(\partial\Omega^*\setminus \BB_r(a),\Omega^*; |t|^\theta\right)\quad {\rm on}\quad \R^N\setminus\left(\Omega\cap B_r(a)\right),
    $$
    we obtain from the comparison principle that on $\left(\Omega\cap B_r(a)\right)^*$ 
    \[
    \begin{split}
        \omega\left(\R^N\setminus B_r(a),\left(\Omega\cap B_r(a)\right)^*; |t|^\theta\right)&=\overline{P}_{\left(\Omega\cap B_r(a)\right)^*}\chi_{\R^N\setminus B_r(a)}\\
        &\geq\overline{P}_{\left(\Omega\cap B_r(a)\right)^*}\omega\left(\partial\Omega^*\setminus \BB_r(a),\Omega^*; |t|^\theta\right).
    \end{split}
    \]
    We know $\Omega$ is a $s$-regular domain in $\R^N$, so $\Omega^*$ is $|t|^\theta$-regular in $\R^{N+1}$ by Lemma \ref{lem-besovcapest}, then every point on $\partial\left(\Omega\cap\overline{B_r(a)}\right)^*$ is $|t|^\theta$-regular boundary point for $\left(\Omega\cap\overline{B_r(a)}\right)^*$ by \cite{BB1}. Since the upper Perron solution is the largest $|t|^\theta$-harmonic solution to the Dirichlet problem with given boundary data, we have
    $$
        \omega\left(\R^N\setminus B_r(a),\left(\Omega\cap B_r(a)\right)^*; |t|^\theta\right)
        \geq
        \omega\left(\partial\Omega^*\setminus \BB_r(a),\Omega^*; |t|^\theta\right)\quad\text{on}\ \left(\Omega\cap B_r(a)\right)^*.
    $$
    Then we get the desired result from assumption $\LHMDs_\sigma$ on $\Omega\cap B_r(a)$ and Theorem \ref{thm-cs-frachm}.

\end{proof}

\subsection{$\GHMDs_\sigma\Rightarrow\LHMDs_\sigma$}

We have proved $\LHMDs_\sigma$ implies $\GHMDs_\sigma$. For the converse part, firstly we need to define {\it uniformly perfect}.
Here we denote $A(x,r_1,r_2)$ be the annulus $B_{r_2}(x)\setminus B_{r_1}(x)$ with center $x$ and radii $0<r_1<r_2$.
\begin{Def}
    Given a set $E\in\R^N$, we say $E$ is {\it uniformly perfect} if there are constants $0<C_5<1$ and $r_0>0$ such that $A(x,C_5r, r)\cap E\neq\emptyset$ for every $x\in E$ and all $0<r<r_0$.
\end{Def}

Here we use two lemmas to conclude the desired result.
\begin{Lemma}\label{lem-gtoperf}
    If $\Omega$ satisfies the $\GHMDs_\sigma$ for some $\sigma\in(0,1)$, then $\partial\Omega$ is uniformly perfect.
\end{Lemma}
\begin{proof}
    Let $a\in\partial\Omega$, and let $0<\rho_1<\rho_2<\diam{\Omega}/2$.
    
    Suppose that $\partial\Omega\cap A(a, \rho_1,\rho_2)=\emptyset$, we prove that $\rho_1/\rho_2\geq\varepsilon>0$.
    W.L.O.G we suppose that $C_4>1$ and $C_4\rho_1\leq\rho_2/(2C_4)$. Since $A(a,\rho_1,\rho_2)\cap\partial\Omega=\emptyset$, we have
    $A(a, C_4\rho_1, \rho_2/C_4)\subset\Omega$. We write $\eta=C_4\rho_1$ and $\gamma=\rho_2/C_4$, then 
    \begin{equation}\label{equ:gtol1}
        A(a,\eta, \gamma)\subset\Omega\subset\Omega^*.
    \end{equation}
    Generally we can choose appropriate $\rho_2$ such that $\partial B_{\rho_2}(a)\cap\partial\Omega=b$. Then by noticing \ref{equ:gtol1} we define
    $$
        E=\overline{\BB_\eta(a)}\setminus\Omega^*=\BB_\gamma(a)\setminus\Omega^*,
    $$
    which is not empty since $a\in\partial\Omega\subset\partial\Omega^*$.
    Then we have 
    \begin{equation}\label{equ:gtol2}
        \rcap_{|t|^\theta}(E,\Omega^*\cup E)\leq \rcap_{|t|^\theta}(\overline{\BB_\eta(a)}, \BB_\gamma(a))\leq C\eta^{N-2s}.
    \end{equation}

    Denote $u_E$ as the potential for condenser $(E,\Omega^*\cup E)$, i.e., $u_E=1$ quasi-everywhere on $E$, and $u_E=0$ quasi-everywhere on $\R^{N+1}\setminus\left(\Omega^*\cup E\right)$.
    Since $\eta\leq \gamma/2$, $A(a,\eta, \gamma)\cap\partial\Omega=\emptyset$ and $b\in\partial\Omega\cap\partial B(a,\rho_2)$, by comparison principle we obtain
    $$
        u_E\leq \omega\left(\partial\Omega^*\setminus B_{\gamma/2}(b),\Omega^*; |t|^\theta\right)\quad {\rm on}\ \Omega, 
    $$
    which implies 
    $$
        u_E(x)\overset{\GHMDs_\sigma}{\leq} C_2\left(\frac{d(x,b)}{\gamma/2}\right)^\sigma, \quad{\rm for}\ x\in\Omega\cap B_{\gamma/2}(b).
    $$
    Then we can choose $\beta=\frac{1}{2(3C_2)^{1/\sigma}}$ such that $u_E\leq 1/3$ on $B_{\beta\gamma}(b)$ since $u_E=0$ on $B_{\gamma/2}(b)\setminus\Omega$.

    As $E=\BB_\gamma(a)\setminus\Omega^*$, $A(a,\eta,\gamma)\subset\Omega$,  by Theorem 11.4 and Corollary 11.8 in \cite{HKM} we obtain that
    $$
        u_E=1-\omega\left(\partial\Omega^*\setminus B_\gamma(a),\Omega^*; |t|^\theta\right)\quad{\rm on}\ \Omega.
    $$
    Since $u_E=1$ on $B_{\beta\gamma}(a)\setminus\Omega\subset B_\gamma(a)\setminus\Omega$, then by $\GHMDs_\sigma$ on $\Omega$ we have
    $$
        u_E\geq 2/3\quad {\rm on} \ B_{\beta\gamma}(a).
    $$
    Define $v=\max\{u_E,1/3\}-1/3$. By the definition of $\beta$ we have
    $$
        \frac{|\{x\in B_{2\rho_2}(a):v(x)=0\}|}{|B_{2\rho_2}(a)|}\geq\frac{|B_{\beta\gamma}(b)|}{|B_{2\gamma}(a)|}\geq C_\beta>0.
    $$
    Then by fractional Poincar\'e inequality (see e.g. Lemma 2.4 in \cite{BLP}) and the definition of $\gamma$ we obtain
    \[
    \begin{split}
        \gamma^s\left(\int_{B_{2\rho_2}(a)}\fint_{B_{2\rho_2}(a)}\frac{|v(x)-v(y)|^2}{|x-y|^{N+2s}}\,dxdy\right)^{1/2}&\geq\left(\fint_{B_{2\rho_2}(a)}v^2\right)^{1/2}\\
        &\geq \left(\fint_{B_{\beta\gamma}(a)}(1/3)^2\right)^{1/2},
    \end{split}
    \]
    which implies
    \[
    \begin{split}
        \rcap_{|t|^\theta}(E,\Omega^*\cup E)&=\int_{\R^{N+1}}|\nabla u_E|^2|t|^\theta\,dxdt\overset{\ref{global-besov}}{\geq}\int_{\R^N}\int_{\R^N}\frac{|u_E(x)-u_E(y)|^2}{|x-y|^{N+2s}}\,dxdy\\
        &\geq\int_{B_{2\rho_2}(a)}\int_{B_{2\rho_2}(a)}\frac{|u_E(x)-u_E(y)|^2}{|x-y|^{N+2s}}\,dxdy\\
        &\geq \int_{B_{2\rho_2}(a)}\int_{B_{2\rho_2}(a)}\frac{|v(x)-v(y)|^2}{|x-y|^{N+2s}}\,dxdy\\
        &\geq C\gamma^{-2s}|B_{2\rho_2}(a)|\geq C\gamma^{N-2s},
    \end{split}
    \]
    which together with \ref{equ:gtol2} yields that $\gamma/\eta$ is bounded, and so is $\rho_2/\rho_1$. 
\end{proof}

\begin{Lemma}
    $\GHMDs_\sigma\Rightarrow\LHMDs_\sigma$.
\end{Lemma}
\begin{proof}
    Let $a\in\partial\Omega$ and $0<r<r_0$. By Lemma \ref{lem-gtoperf}, we know $\GHMDs_\sigma$ implies uniformly perfectness of $\partial\Omega$, then we can find some $\rho>0$ such that $S_\rho(a)\cap\partial\Omega\neq\emptyset$ with $C_5r\leq\rho<r$.
    
    Let $c$ be a small positive number to be determined later. 
    It's obvious to see that we can find finitely many points $z_1,\ z_2,\ ...,\ z_n\in A(a,\rho/C_4,C_4\rho)\subset\R^N$ such that $S_\rho(a)\subset\bigcup^n_{j=1}\BB_{cr}(z_j)$, and $\{\BB_{cr}(z_j)\}$ form a chain, which is defined as, for $\forall k,\ell\in\{1,\ 2,\ ...,\ n\}$, there is a sub-collection of balls $\{\BB_1,\ \BB_2,\ ...,\ \BB_t\}\subset\{\BB_{cr}(z_j)\}$ such that $\BB_{cr}(z_k)=\BB_1$ and $\BB_{cr}(z_\ell)=\BB_t$, also $\BB_i\cap \BB_{i+1}\neq\emptyset$ for $\forall i\in\{1,\ 2,\ ...,\ t\}$. Here $n$ only depends on $c$ and $N$.

    By scaling we easily see that
    \begin{equation}\label{equ-gtol4}
        \begin{split}
            \bigcup^n_{j=1}B_{4cr}(z_j)&\subset A\left(a, \frac{\rho}{C_4}-4cr, C_4\rho+4cr\right)\\
            &\subset A\left(a, \left(\frac{C_5}{C_4}-4c\right)r, \left(C_4+4c\right)r\right).
        \end{split}
    \end{equation}
    Now let $c$ be small enough such that 
    \begin{equation}\label{equ-gtol5}
        4c\leq\frac{C_5}{2C_4}:=\eta.
    \end{equation}
    Define
    \begin{equation}\nonumber
        u=\left\{
        \begin{split}
            \omega\left(\partial\Omega^*\cap \BB_{\eta r}(a), \Omega^*; |t|^\theta\right)&\quad{\rm on}\ \Omega^*,\\
            0\qquad\qquad&\quad{\rm on}\ \R^{N+1}\setminus\Omega^*.
        \end{split}
        \right.
    \end{equation}
    
    Then $0\leq u\leq 1$ on $\R^{N+1}$ is a $|t|^\theta$-subminimizer in $\R^{N+1}\setminus\overline{\BB_{\eta r}(a)}\supsetneqq\bigcup^n_{j=1}\BB_{4cr}(z_j)$ by \ref{equ-gtol4} and \ref{equ-gtol5}. 
    We may assume and fix one point $z\in\partial\Omega\cap S_\rho(a)$ such that $z\in B_{cr}(z_1)$ without loss of generality. Obviously $\BB_{3cr}(z)\subset \BB_{4cr}(z_1)\subsetneqq\R^{N+1}\setminus\overline{\BB_{\eta r}(a)}$, then it follows from the comparison principle that
    $$
        u\leq\omega\left( \partial\Omega^*\setminus \BB_{3cr}(z), \Omega^*; |t|^\theta\right)\quad{\rm on}\ \Omega^*.
    $$
    Then by $\GHMDs_\sigma$ property on $\Omega$ we have
    $$
        u\leq 1/2-\varepsilon\quad{\rm on} \ \BB_{\beta r}(z)\cap\Omega
    $$
    for some $\beta>0$ independent of $a$ and $r$ and for small enough $\forall\varepsilon>0$ depending on $\beta$.

    In fact by Theorem 11.6 in \cite{HKM} and Remark \ref{rem-har-mini} we know $u=\omega\left( \partial\Omega^*\cap \BB_{\eta r}(a), \Omega^*; |t|^\theta\right)$ on  $\bigcup^n_{j=1}\BB_{4cr}(z_j)$, hence $|t|^\theta$-harmonic on $\bigcup^n_{j=1}\BB_{4cr}(z_j)$. Then we know that
    $$
        u, -u\in DG_{|t|^\theta}\left(\bigcup^n_{j=1}\BB_{4cr}(z_j)\right).
    $$
    Hence by Lemma \ref{lem-holderdg} we can make $\beta$ small enough such that
    $$
        \max\limits_{\BB_{\beta r}(z)} u\leq 1/2-\varepsilon+\mathop{\rm osc}\limits_{\BB_{\beta r}(z)}u\leq 1/2
    $$
    based on the fact $0\leq u\leq 1$ on $\R^{N+1}$. 
    Then by choosing $\rho_1=2cr$ it follows from Lemma \ref{lem-decaydg} that $u\leq 1-\varepsilon_1$ on $\BB_{cr}(z_1)$ for some $\varepsilon_1$ independent of $a$ and $r$.
    Since $\bigcup^n_{j=1}\BB_{cr}(z_j)$ is a chain, we can always find some ball, say $\BB_{cr}(z_2)$, intersecting $\BB_{cr}(z_1)$. Then Corollary \ref{col-chaindecay} yields $u\leq 1-\varepsilon_2$ on $\BB_{cr}(z_2)$ for some $\varepsilon_2>0$. Because of finite cover of $S_\rho(a)$ by $\{\BB_{cr}(z_j)\}$, by repeating this process finite times we eventually obtain $u\leq 1-\varepsilon_0$ on $\bigcup^n_{j=1}\BB_{cr}(z_j)$ for some $\varepsilon_0>0$ which is independent of $a$ and $r$. Especially we have $u\leq 1-\varepsilon_0$ on $S_\rho(a)$. See Fig. \ref{fig:chain}.

    Since by Proposition 9.31 and Theorem 11.4 in \cite{HKM} we have
    $$
        \omega\left( \partial\Omega^*\cap \BB_{\eta r}(a), \Omega^*; |t|^\theta\right)+\omega\left( \partial\Omega^*\setminus \BB_{\eta r}(a), \Omega^*; |t|^\theta\right)=1\quad{\rm on} \ \Omega^*,
    $$
    and hence together with the fact that
    $$
        \frac{C_5}{2C_4}r=\eta r<C_5r\leq\rho,
    $$ 
    we obtain $\omega\left( \partial\Omega^*\setminus\BB_{\eta r}(a), \Omega^*; |t|^\theta\right)\geq\varepsilon_0$ on $\Omega^*\cap \s_\rho(a)$. Then by comparison principle, we obtain
    $$
        \frac{1}{\varepsilon_0}\omega\left( \partial\Omega^*\setminus \BB_{\eta r}(a), \Omega^*; |t|^\theta\right)
        \geq \omega\left(\R^N\setminus B_\rho(a), (\Omega\cap B_\rho(a))^*; |t|^\theta\right)\quad{\rm on}\ \Omega^*\cap \BB_{\rho}(a).
    $$
    Then by $\GHMDs_\sigma$ we have that for all $x\in\Omega\cap B_\rho(a)$
    \[
    \begin{split}
        \omega\left( \R^N\setminus B_r(a), (\Omega\cap B_r(a))^*; |t|^\theta\right)&\leq \omega\left(\R^N\setminus B_\rho(a), (\Omega\cap B_\rho(a))^*; |t|^\theta\right)\\
        &\leq \frac{C_2}{\varepsilon_0}\left(\frac{d(x,a)}{\eta r}\right)^\sigma.
    \end{split}
    \]
    For the inequality on the part $\Omega\cap\left(B_r(a)\setminus B_\rho(a)\right)$, we just notice the facts that 
    $\rho\geq C_5r$,
    and
    $$
        \omega\left(\R^N\setminus B_r(a), (\Omega\cap B_r(a))^*; |t|^\theta\right)\leq 1\quad{\rm on}\ \Omega\cap\left(B_r(a)\setminus B_\rho(a)\right),
    $$
    then for $x\in\Omega\cap\left(B_r(a)\setminus B_\rho(a)\right)$, we obtain
    \[
    \begin{split}
        \omega\left(\R^N\setminus B_r(a), (\Omega\cap B_r(a))^*; |t|^\theta\right)
        \leq\left(\frac{d(x,a)}{d(x,a)}\right)^\sigma
        \leq\left(\frac{d(x,a)}{\rho}\right)^\sigma
        \leq\left(\frac{d(x,a)}{C_5r}\right)^\sigma,
    \end{split}
    \]
    which yields the desired decay on $\Omega\cap\left(B_r(a)\setminus B_\rho(a)\right)$.
    Therefore the $\LHMDs_\sigma$ property follows.

    \begin{figure}[t]
        \centering
        \includegraphics[scale=0.6]{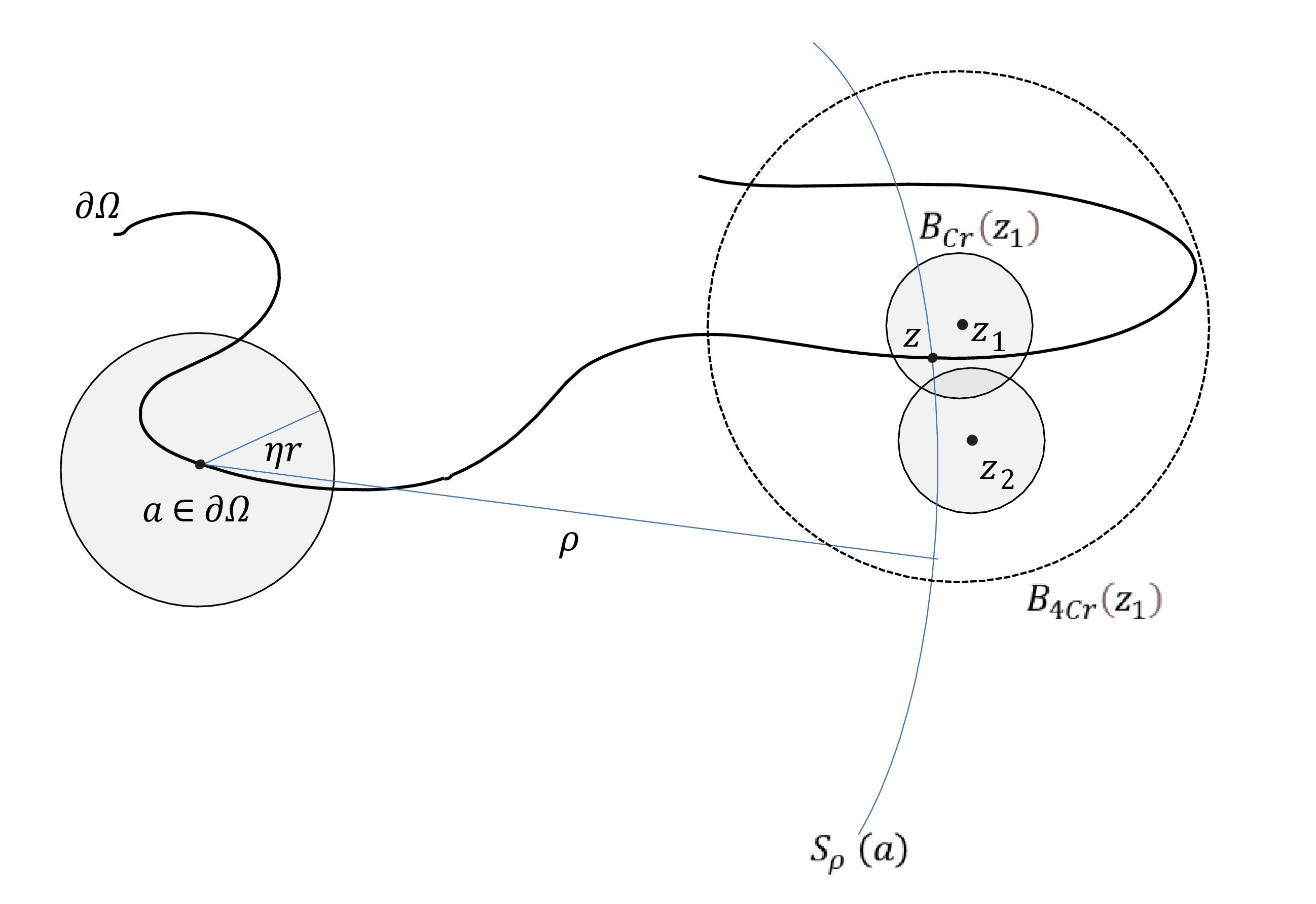}
        \caption{$S_\rho(a)\subset\bigcup^n_{j=1}\BB_{cr}(z_j)$ and $u\in DG_{|t|^\theta}\left(\bigcup^n_{j=1}\BB_{4cr}(z_j)\right)$.}
        \label{fig:chain}
    \end{figure}

\end{proof}

\section{Proof of Theorem \ref{thm:main2}}
We only need to prove the following.
\begin{Lemma}\label{lem-capden-lhmd}
    $\LHMDs_\sigma$ holds on $\Omega$ for some $\sigma>0$ if and only if $\R^N\setminus\Omega$ is uniformly $s$-fat.
\end{Lemma}

For preparation, we need the following lemma, which is a weighted invariant of Lemma 6.2 in \cite{AS} (see also Lemma 2.16 in \cite{HKM}). For readers' convenience, we give it below.

\begin{Lemma}\label{lem-capest}
    Let $x_0\in\R^{N+1}$, $E\subset \overline{B_r(x_0)}\subset\R^{N+1}$ and $0<r\leq r_0$.
    \begin{itemize}
        \item[(i)] If $0<\ell\leq 1$, then
        $$
            \rcapw\left(\overline{B_{\ell r}(x_0)}, B_{2r}(x_0)\right)\leq\rcapw\left(\overline{B_{r}(x_0)}, B_{2r}(x_0)\right)\leq\C\rcapw\left(\overline{B_{\ell r}(x_0)}, B_{2r}(x_0)\right),
        $$
        where $\C$ depends only on $\ell$.
        \item[(ii)] Let $\gamma\geq 1$, then
        $$
            \rcapw\left(E, B_{2\gamma r}(x_0)\right)\leq\rcapw\left(E, B_{2r}(x_0)\right)\leq\C\rcapw\left(E, B_{2\gamma r}(x_0)\right),
        $$
        where $\C$ depends only on $\gamma$.
    \end{itemize}
\end{Lemma}

For $a\in\R^N$, $E\subset\R^N$, and $r>0$, we denote
$$
    \psi(a, E, r)=\frac{\rcapb\left(E\cap B_r(a), B_{2r}(a)\right)}{\rcapb\left(B_r(a), B_{2r}(a)\right)}.
$$
Then the {\it uniform fractional $p$-factness} of $E$ means $\psi(a, E, r)\geq C>0$ for $a\in E$ and $0<r<r_0$. The following lemma states that it's sufficient to prove the uniform fractional $p$-factness of boundary points $\partial E$ to conclude the uniform fractional $p$-factness of $E$.
\begin{Lemma}\label{lem-fatbdr}
    If $\psi(x, E, r) \geq \C$ for every $x\in\partial E$ and $0<r<r_0$, then $E$ is uniformly fractional $p$-fat.
\end{Lemma}
\begin{proof}
    The constant $\C$ below is generic and depends only on $N$, $p$, and $s$.
    
    Let's suppose that $a\in E$ is an arbitrary interior point. We only need to prove that $\psi(a, E, r)\geq \C>0$.
    Let $\rho_0=d(a, \R^N\setminus E)>0$, then we can find $b\in\partial E$ such that $\rho_0=d(a, b)$.
    Then we amount to the following two cases.
    
    \quad{\it Case 1:} $r< 2\rho_0$. Then $\overline{B_{r/2}(a)}\subset E$. In view of Lemma \ref{lem-capest} we have
    \[
    \begin{split}
        \psi(a, E, r)&\geq
        \frac{\rcapb\left(\overline{B_{r/2}(a)}, B_{2r}(a)\right)}{\rcapb\left(\overline{B_{r}(a)}, B_{2r}(a)\right)}
        \overset{\text{Lemma}\ \ref{lem-besovcapest}}{\geq}\frac{\rcapw\left(\overline{\BB_{r/4}(a)}, \BB_{2r}(a)\right)}{\rcapw\left(\overline{\BB_{r}(a)}, \BB_{2r}(a)\right)}\geq \C>0.
    \end{split}
    \]

    \quad{\it Case 2:} $r\geq2\rho_0$. Then we have 
    $\overline{B_{r/2}(b)}\subset\overline{B_{r}(a)}\subset\overline{B_{3r/2}(b)}$, and $\overline{B_{r}(b)}\subset\overline{B_{2r}(a)}\subset\overline{B_{5r/2}(b)}$.
    In view of Lemma \ref{lem-capest} we have
    \[
    \begin{split}
        \rcapb\left(E\cap\overline{B_r(a)}, B_{2r}(a)\right)
        &\overset{\text{Lemma}\ \ref{lem-equivalentcap}}{\simeq}
        \rcapw\left((E\cap\overline{B_r(a)})\times\{0\}, \BB_{2r}(a)\right)\\
        &\geq\rcapw\left((E\cap\overline{B_{r/2}(b)})\times\{0\}, \BB_{2r}(a)\right)\\
        &\geq\rcapw\left((E\cap\overline{B_{r/2}(b)})\times\{0\}, \BB_{5r/2}(b)\right)\\
        &\geq \C\rcapw\left((E\cap\overline{B_{r/2}(b)})\times\{0\}, \BB_{r}(b)\right)\\
        &\simeq\rcapb\left(E\cap\overline{B_{r/2}(b)}, B_{r}(b)\right).
    \end{split}
    \]
    In the meanwhile, we also have
    \[
    \begin{split}
        \rcapb\left(\overline{B_r(a)}, B_{2r}(a)\right)&\overset{\text{Lemma}\ \ref{lem-besovcapest}}{\leq}\rcapw\left(\overline{\BB_r(a)}, \BB_{2r}(a)\right)
        \leq\rcapw\left(\overline{\BB_{3r/2}(b)}, \BB_{2r}(a)\right)\\
        &\overset{\text{Lemma}\ \ref{lem-capest}}{\leq} \C\rcapw\left(\overline{\BB_{r/2}(b)}, \BB_{2r}(b)\right)\leq \C \rcapw\left(\overline{\BB_{r/2}(b)}, \BB_{r}(b)\right)\\
        &\overset{\text{Lemma}\ \ref{lem-capest}}{\leq}\C\rcapw\left(\overline{\BB_{r/3}(b)}, \BB_{r}(b)\right)\\
        &\overset{\text{Lemma}\ \ref{lem-besovcapest}}{\leq}\C \rcapb\left(\overline{B_{r/2}(b)}, B_{r}(b)\right).
    \end{split}
    \]
    Therefore, we have
    $$
        \psi(a, E, r)\geq \C\psi(b, E, r/2)\geq \C>0,
    $$
    which yields the desired results.
    
\end{proof}

\begin{Lemma}\label{lem-fatequal}
    Denote by $\Omega^c$ as $\R^N\setminus\Omega$.
    Let $0<r<r_0$, and let $a\in\Omega^c$. Then the uniformly fractional $p$-fat of $\R^N\setminus\Omega$ in $\R^N$ is equivalent to the uniformly $|t|^\theta$-weighted $p$-fat of $\partial\Omega^*$ in $\R^{N+1}$, that is,
    $$
        \frac{\rcapb\left(\Omega^c\cap B_r(a), B_{2r}(a)\right)}{\rcapb\left(B_r(a), B_{2r}(a)\right)}\geq \C_1\Leftrightarrow
        \frac{\rcapw\left(\partial\Omega^*\cap \BB_r(a), \BB_{2r}(a)\right)}{\rcapw\left(\BB_r(a), \BB_{2r}(a)\right)}\geq \C_2,
    $$
    where $\C_1$ and $\C_2$ depend only on $N$, $s$, and $p$.
\end{Lemma}
\begin{proof}
    It's obvious that uniformly $|t|^\theta$-weighted $p$-fat in $\R^{N+1}$ implies the uniformly fractional $p$-fat in $\R^N$ by Lemma \ref{lem-equivalentcap}, Lemma \ref{lem-besovcapest}, and the fact that $\partial\Omega^*\cap\BB_r(a)=(\Omega^c\cap B_r(a))\times\{0\}$.

    Let $\ell\in(0,1)$ be fixed. Then in the other direction, we notice that
    \[
    \begin{split}
        \rcapw\left(\partial\Omega^*\cap\BB_r(a), \BB_{2r}(a)\right)
        &\geq\rcapw\left(\partial\Omega^*\cap\BB_{\ell r}(a), \BB_{2r}(a)\right)\\
        &\overset{\text{Lemma}\ \ref{lem-capest}}{\simeq}
        \rcapw\left((\partial\Omega^*\cap\BB_{\ell r}(a), \BB_{2\ell r}(a)\right)\\
        &\overset{\text{Lemma}\ \ref{lem-equivalentcap}}{\simeq}
        \rcapb\left(\Omega^c\cap B_{\ell r}(a), B_{2\ell r}(a)\right).
    \end{split}
    \]
    In the meanwhile, we have
    \[
    \begin{split}
        \rcapw\left(\BB_r(a), \BB_{2r}(a)\right)&\overset{\text{Lemma}\ \ref{lem-capest}}{\leq}\C\rcapw\left(\BB_{\ell r}(a), \BB_{2r}(a)\right)\\
        &\overset{\text{Lemma}\ \ref{lem-capest}}{\leq}\C\rcapw\left(\BB_{\ell r}(a), \BB_{2\ell r}(a)\right)\\
        &\overset{\text{Lemma}\ \ref{lem-equivalentcap}}{\simeq}\rcapb\left(B_{\ell r}(a), B_{2\ell r}(a)\right).
    \end{split}
    \]
    Then by combining the two estimates above we get the other direction.
\end{proof}
 
We also need the following estimates on the modulus of continuity of solutions to Dirichlet problems. For the classical case one can refer to \cite{Mazya2}, and for the non-local case, one can refer to the recent result in \cite{Li}.
\begin{Lemma}[\cite{AS}, Lemma 6.4]\label{lem-mod-cont}
    Let $a\in\partial\Omega$, let $r>0$ be fixed. Let $u$ be the $p$-potential for $\overline{\BB_r(a)}\setminus\Omega^*$ with respect to $\BB_{5r}(a)$. Then
    $$
        1-u(x)\leq\exp\left(-C\int^r_\rho\psi(a,\R^{N+1}\setminus\Omega^*,t)^{1/(p-1)}\frac{\,dt}{t}\right)\quad\text{for}\ 0<\rho\leq r\ \text{and}\ x\in \BB_\rho(a).
    $$
\end{Lemma}

\medskip

Now we give the proof of Lemma \ref{lem-capden-lhmd}.
\begin{proof}[Proof of Lemma \ref{lem-capden-lhmd}]
    Firstly let's suppose that $\R^N\setminus\Omega$ is uniformly $s$-fat, then $\left(\R^N\setminus\Omega\right)\times\{0\}$ is uniformly $|t|^\theta$-fat in $\R^{N+1}$ by Lemma \ref{lem-fatequal}. Let $a\in\partial\Omega\subset\partial\Omega^*$, $0<r<r_0$. 
    Let $u$ be the potential for $\overline{\BB_{r/5}(a)}\setminus\Omega^*$ with respect to $\BB_r(a)$. Hence $u=0$ q.e. on $\R^{N+1}\setminus\BB_r(a)$. Then by the comparison principle, we have
    $$
        \omega\left(\R^N\setminus B_r(a), (\Omega\cap B_r(a))^*; |t|^\theta\right)\leq 1-u\quad\text{on}\ \Omega^*\cap \BB_r(a),
    $$
    and then by Lemma \ref{lem-mod-cont} we can forward as
    \[
    \begin{split}
        \omega\left(\R^N\setminus B_r(a), (\Omega\cap B_r(a))^*; |t|^\theta\right)&\leq 1-u(x)\\
        &\leq C\left(\frac{\rho}{r/5}\right)^\delta\quad\text{for}\ x\in \BB_\rho(a)\ \text{and}\ 0<\rho\leq r/5, 
    \end{split}
    \]
    in which $\delta$ depends only on $C_0$ (see Definition \ref{def:capdensity}). Then we get that $\LHMD^{|t|^\theta}_\sigma$ holds on $\BB_{r/5}(a)\cap\Omega^*$, and so $\LHMDs_\sigma$ holds on $B_{r/5}(a)\cap\Omega$.

    Now let's suppose that $\LHMDs_\sigma$ holds on $\Omega$ for some $\sigma>0$. Thanks to Lemma \ref{lem-fatbdr} we just need to prove that $\psi(a, \R^N\setminus\Omega, r)\geq C>0$ holds for every $a\in\partial\Omega$ and $0<r<r_0$. Now let's fix $a\in\partial\Omega$, and some $r\in(0,r_0)$. Then let function $v$ be the $|t|^\theta$-potential function for $\overline{\BB_r(a)}\setminus\Omega^*$ (i.e. $\overline{B_r(a)}\setminus\Omega$) with respect to $\BB_{2r}(a)$. Then by the comparison principle we have
    $$
        \omega\left(\R^N\setminus B_r(a), (B_r(a))^*; |t|^\theta\right)\geq 1-v\quad\text{on}\ \Omega^*\cap \BB_r(a)\supsetneqq\Omega\cap B_r(a),
    $$
    and then by $\LHMDs_\sigma$, we have
    for some $C_8>1$ such that
    $$
        \omega\left(\R^N\setminus B_r(a), (B_r(a))^*; |t|^\theta\right)\leq 1/2 \quad\text{on}\ \Omega\cap\overline{B_{r/C_8}(a)}.
    $$
    Consequently, we have $v\geq 1/2$ on $\Omega\cap\overline{B_{r/C_8}(a)}$. As $v=1$ quasi-everywhere on $\overline{\BB_r(a)}\setminus\Omega^*$, we then have $v\geq 1/2$ quasi-everywhere on $\overline{B_{r/C_8}(a)}$. Hence $2v$ is an admissible function for the relative capacity $\rcap_{|t|^\theta}\left(\overline{B_{r/C_8}(a)}\times\{0\}, \BB_{2r}(a)\right)$. So we have
   \[
   \begin{split}
       \rcap_{B^s_2}\left(\overline{B_r(a)}\setminus\Omega, B_{2r}(a)\right)&\overset{\text{Lemma}\ \ref{lem-equivalentcap}}{\simeq}\rcap_{|t|^\theta}\left(\left(\overline{B_r(a)}\setminus\Omega\right)\times\{0\}, \BB_{2r}(a)\right)\\
       &=\int_{\BB_{2r}(a)}|\nabla (2v)|^2|t|^\theta\,dxdt\\
       &\geq 4\rcap_{|t|^\theta}\left(\overline{B_{r/C_8}(a)}\times\{0\}, \BB_{2r}(a)\right)\\
       &\overset{\text{Lemma}\ \ref{lem-equivalentcap}}{\simeq}C^\prime\rcap_{B^s_2}\left(\overline{B_{r/C_8}(a)}, B_{2r}(a)\right)\\
       &\overset{\text{Lemma}\ \ref{lem-besovcapest}}{\geq}C^\prime\rcap_{|t|^\theta}\left(\overline{\BB_{r/(2C_8)}(a)}, \BB_{2r}(a)\right).
   \end{split}
   \]
    Therefore, we have 
    \[
    \begin{split}
        \psi(a, \R^N\setminus\Omega,r)=\frac{\rcap_{B^s_2}\left(\overline{B_r(a)}\setminus\Omega, B_{2r}(a)\right)}{\rcap_{B^s_2}\left(\overline{B_r(a)}, B_{2r}(a)\right)}
        &\overset{\text{Lemma}\ \ref{lem-besovcapest}}{\geq}
        C^\prime\frac{\rcap_{|t|^\theta}\left(\overline{\BB_{r/(2C_8)}(a)}, \BB_{2r}(a)\right)}{\rcap_{|t|^\theta}\left(\overline{\BB_r(a)}, \BB_{2r}(a)\right)}\\
        &\geq C>0,
    \end{split}
    \]
    which yields the desired result.
    
\end{proof}

Now we are in the right position to give the proof of Proposition \ref{prop-lebeguetocap}.
\begin{proof}
    In fact, the proof is just a combination of Corollary 2.6 in \cite{AS}, Lemma \ref{lem-fatequal}, and Theorem \ref{thm:main2}.
\end{proof}

\medskip


\end{document}